\documentclass{amsart}
\usepackage{tikz}
\usetikzlibrary{shapes}
\usepackage{amsmath, amssymb, amsthm, amsfonts, mathrsfs, amsfonts}
\usepackage[a4paper,left = 2.5cm, right= 2.5cm, top = 2.5cm, bottom = 2.5cm]{geometry}
\usepackage{cases}
\newcommand{\N}{\ensuremath{\mathbb{N}}}
\newcommand{\Z}{\ensuremath{\mathbb{Z}}}
\newcommand{\R}{\ensuremath{\mathbb{R}}}
\newcommand{\C}{\ensuremath{\mathbb{C}}}
\newtheorem{thm}{Theorem}[section]
\newtheorem{condition}[thm]{Condition}
\newtheorem*{thm2}{Theorem}
\newtheorem*{question*}{Question}
\newtheorem{question}{Question}
\newtheorem{cor}[thm]{Corollary}
\newtheorem*{cor2}{Corollary}
\newtheorem{lemma}[thm]{Lemma}
\newtheorem{prop}[thm]{Proposition}
\theoremstyle{remark}
\newtheorem{rmk}[thm]{Remark}

\theoremstyle{definition}
\newtheorem{defin}[thm]{Definition}

\newtheorem{ex}[thm]{Example}
\renewcommand{\l}{\ensuremath{\mathcal{L}}}
\newcommand{\+}[1]{\ensuremath{\overset{{#1}}{+}}}

\newcommand{\on}{\operatorname}

\begin{document}

\title{Existence of symmetric maximal noncrossing collections of $k$-element sets}
\author{Andrea Pasquali}
\address{Dept.~of Mathematics, Uppsala University, P.O.~Box 480, 751 06 Uppsala, Sweden}
\email{andrea.pasquali@math.uu.se}
\author{Erik Th\"ornblad}
\email{erik.thornblad@math.uu.se}
\author{Jakob Zimmermann}
\email{jakob.zimmermann@math.uu.se}

\begin{abstract}
  We investigate the existence of maximal collections of mutually noncrossing $k$-element subsets of $\left\{ 1, \dots, n \right\}$ that
  are invariant under adding $k\pmod n$ to all indices. Our main result is that such a collection exists if and only if $k$ is congruent to $0, 1$ or $-1$ modulo
  $n/\on{GCD}(k,n)$. Moreover, we present some algebraic consequences of our result related to self-injective Jacobian algebras.
\end{abstract}

\maketitle

\section*{Introduction}
Two subsets $I$ and $J$ of $\left\{ 1, \dots, n \right\}$ are said to be \emph{noncrossing} if there are no cyclically ordered $a, b, c, d$ such 
that $a,c\in I\setminus J$ and $b,d\in J\setminus I$. We are interested in maximal collections of mutually noncrossing sets of some fixed 
size $k$. In the case $k=2$, such collections are maximal collections of noncrossing segments between $n$ points on a circle, that is, 
triangulations of the $n$-gon. 
The case of general $k$ can be tackled using the machinery of alternating strand diagrams and plabic graphs developed by Postnikov \cite{Pos06}.

We call a collection of $k$-element subsets of $\left\{ 1, \dots, n \right\}$ \emph{symmetric} if it is invariant under adding $k\pmod n$ to all indices.
In a recent paper \cite{Pas17b} Pasquali showed that symmetric maximal noncrossing collections naturally give rise to
self-injective Jacobian algebras. More precisely he showed that any maximal symmetric noncrossing collection for a pair $(k,n)$ gives rise to a self-injective Jacobian algebra
whose quiver with potential can be obtained from an embedding of the collection into the plane. Now a natural
question is for which pairs $(k,n)$ does there exist such a collection. In this paper we answer this question by the following theorem.
\begin{thm2}[Theorem~\ref{thm:main}]
  Let $(k,n)\in \Z^2$, with $n\geq 1$ and $0\leq k\leq n$, and call $d = n/\on{GCD}(k,n)$. Then the following are equivalent:
  \begin{itemize}
    \item there exists a symmetric maximal noncrossing collection of $k$-element subsets of $\left\{ 1, \dots, n \right\}$;
    \item the number $k$ is congruent to $0, 1$ or $-1$ modulo $d$.
  \end{itemize}
\end{thm2}

Our proof goes via an explicit construction of a symmetric maximal noncrossing collection. It should be noted that there exist symmetric maximal noncrossing
collections which do not arise in this way. In particular the problem of classifying all symmetric maximal noncrossing collections is, as far as we know, still open.

Our motivation for studying maximal noncrossing collections comes as mentioned above from algebra, specifically representation theory
and cluster theory. In the following we give a more detailed account of the connection between the combinatorics 
of maximal noncrossing collections and algebra. It turns out that the combinatorics of noncrossing sets corresponds to the cluster 
combinatorics of the homogeneous coordinate ring of the Grassmannian $\on{Gr}_\C(k,n)$; see \cite{Sco06}. Every cluster consisting of
Pl\"ucker coordinates corresponds to a maximal noncrossing collection, and both the cluster variables and the quiver corresponding 
to the cluster can be constructed from the collection (see \cite{OPS15}). Moreover, there is a categorification of this cluster structure 
using Cohen-Macaulay modules over an infinite dimensional algebra $B = B(k,n)$ \cite{JKS16}. An indecomposable Cohen-Macaulay $B$-module is associated 
to every $k$-element subset of $\{1, \dots, n\}$, and the noncrossing condition corresponds to the vanishing of $\on{Ext}^1_B$ between 
these modules. Thus a maximal noncrossing collection corresponds to a cluster tilting object in the category $\on{CM}(B)$ of Cohen-Macaulay $B$-modules. It was shown 
in \cite{BKM16} that in fact the endomorphism ring of this cluster tilting module is a frozen Jacobian algebra. 
The cluster of Pl\"ucker coordinates corresponding to the collection gives the quiver with potential of this algebra.

One can also consider the analogous story in the stable category $\underline{\on{CM}}(B)$, which corresponds to taking a quotient by 
the idempotent corresponding to the frozen vertices. Then one gets a finite-dimensional Jacobian algebra $\Lambda$, and using the 
triangulated structure of $\underline{\on{CM}}(B)$ one can prove that $\Lambda$ is self-injective if and only if the corresponding 
maximal noncrossing collection is invariant under adding $k$ to all indices $\pmod n$. The operation of adding $k$ modulo $n$ is 
nothing but a planar rotation of the quiver (or of the boundary $n$-gon) by an angle of $\frac{2\pi k }{n}$. One reason to look at self-injective Jacobian 
algebras is that they are precisely the 3-preprojective algebras of 2-representation finite algebras \cite{HI11b}.

These algebras have an automorphism called the Nakayama automorphism. One can consider the order of this automorphism (when it is finite), and all known examples have either small order or 
are in some sense very simple (roughly speaking, the number of vertices is linear in the order). A consequence of our result is:
\begin{cor2}
  [{Corollary~\ref{cor:algcons2}}]
  Let $d\in \Z_{>1}$. There exist infinitely many self-injective Jacobian algebras with a Nakayama automorphism of order $d$. 
\end{cor2}

The paper is organised in the following way. The next section introduces the problem and gives the statement of our main theorem.
In Section \ref{sec:necessity} we prove one direction of Theorem~\ref{thm:main}, using planar embeddings of noncrossing collections.
Section \ref{sec:setup} is devoted to setting up the notation for our constructions and proving some auxiliary results.
In Sections \ref{sec:specialconstruction} and \ref{sec:construction} we construct an explicit symmetric maximal noncrossing collection, 
first if $k\mid n$ and then for general $(k,n)$. This provides the other direction in the proof of Theorem~\ref{thm:main}.
In Section \ref{sec:examples} we compute an explicit example to illustrate our construction. Finally, in Section \ref{sec:algcons} we 
explain some algebraic consequences of our combinatorial result, in particular about existence of self-injective Jacobian algebras.

\subsection*{Acknowledgements.} We would like to express our gratitude to Martin Herschend and Laertis Vaso, for carefully reading the manuscript and suggesting improvements. We also thank two anonymous referees for their feedback.

\section{The problem}\label{sec:problem}

In the following let $0 < k \leq n$ be integers. A \emph{cyclically ordered set} is a finite set $X$ together with a bijection $S_X:X\to X$ 
such that for all $x, y\in X$ there is $n\in \Z$ such that $S_X^n(x) = y$. We think of $S_X$ as a ``successor function''.
If $X$ is a cyclically ordered set and $\varnothing \neq Q\subseteq X$, there is an induced cyclic order $S_Q$ on $Q$. Indeed, for $q\in Q$,
we define $S_Q(q)= S_X^m(q)$, where $m>0$ is 
the least positive integer such that $S_X^m(q)\in Q$. 
If $X$ is cyclically ordered and $x_1,x_2, \dots, x_n\in X$ are distinct elements, we write
\begin{align*}
  x_1<_\circ x_2<_\circ \cdots <_\circ x_n
\end{align*}
if $S_Q(x_i) = x_{i+1}$ for $1\leq i<n$, where $Q = \left\{ x_1, x_2, \dots, x_n \right\}$. 
With this we can now give the following definition:

\begin{defin}
  Let $X$ be a cyclically ordered set.
  Two subsets $I, J\subseteq X$ are said to be \emph{crossing} if there exist $a<_\circ b<_\circ c<_\circ d\in X$ such that $a,c\in I\setminus J$ and 
  $b,d\in J\setminus I$. Otherwise $I, J$ are said to be \emph{noncrossing}.
\end{defin}

For $n \in \N$, we write $[n] = \left\{ 1, 2, \dots, n \right\}$.
For $a,b\in [n]$, we write $$
a\+{n}b = 
\begin{cases}
  a+b, &\text{ if } a+b\in [n];\\
  a+b -n, &\text{ otherwise}.
\end{cases}
$$
We will always consider the cyclic order on $[n]$ given by $S_{[n]}(x) = x\+{n}1$.
For a subset $I$ of $[n]$, we denote by $I\+{n}k$ the subset $\left\{ i\+{n}k \; |\; i\in I \right\}$ of $[n]$.
For a collection $\mathbb I$ of subsets of $[n]$, we denote by $\mathbb I \+{n}k$ the collection $\left\{ I\+{n}k\;|\; I\in \mathbb I \right\}$.
\begin{defin}
  A collection $\mathbb I$ of $k$-element subsets of $[n]$ is a \emph{$(k,n)$-noncrossing collection} if $I$ and $J$ are noncrossing for all $I, J\in \mathbb I$.
  A $(k,n)$-noncrossing collection is \emph{maximal} if it is maximal with respect to inclusion in the set of all $(k,n)$-noncrossing collections.
\end{defin}
We will often omit the reference to $(k,n)$ when it is clear from the context.

\begin{defin}
  A collection $\mathbb I$ of $k$-element subsets of $[n]$ is \emph{symmetric} if $\mathbb I = \mathbb I\+{n}k$.
\end{defin}

Observe that this definition depends on $k$ and $n$. In case these are ambiguous we will spell out $\mathbb I = \mathbb I\+{n}k$.

We can now formally state the question which we address in this paper.

\begin{question}
  \label{question:main}
   For which $(k,n)$ does there exist a maximal $(k,n)$-noncrossing collection which is symmetric?
\end{question}

\begin{rmk}
 It is worth pointing out that we look for symmetric collections which are maximal among all collections, not just among the symmetric ones. 
\end{rmk}

It is easy to see that such collections do not exist for all choices of $n$ and $k$. It turns out that the following condition is what we need.

\begin{condition}\label{cond:star}
The pair $(k,n)\in \Z^2$ is such that $n \geq 1$, $0 \leq k \leq n$, 
and $k$ is congruent to $0, 1$ or $-1$ modulo $n/\on{GCD}(k,n)$.
\end{condition}
Indeed, we have the following:
 \begin{thm} \label{thm:main}
  There exists a symmetric maximal $(k,n)$-noncrossing collection if and only if $(k,n)$ satisfies Condition~$\ref{cond:star}$.
\end{thm}

We give a proof of Theorem \ref{thm:main} at the end of Section \ref{sec:construction}. The strategy is as follows: in Section \ref{sec:necessity} we
will prove that Condition~$\ref{cond:star}$ is necessary, and in Sections \ref{sec:specialconstruction} and \ref{sec:construction} we will explicitly construct a symmetric 
maximal noncrossing collection to show that Condition~$\ref{cond:star}$ is sufficient.

\begin{rmk}
 Observe that $I, J \subseteq [n]$ are noncrossing if and only if the complements $[n] \setminus I$ and $[n] \setminus J$ are. Moreover, a pair $(k,n)$
 satisfies Condition~$\ref{cond:star}$ if and only if the pair $(n-k,n)$ does. Thus it is enough to study the case $k\leq \frac{n}{2}$. Our
 construction and 
 result work for general $k\leq n$, so we do not make this assumption.
 \end{rmk}

 To prove Theorem \ref{thm:existence} we will in fact use a characterisation of maximal noncrossing collections, which was first conjectured in \cite{Sco05} (see also \cite{LZ98})
 and then proved in \cite{OPS15}.

\begin{thm}
  [{\cite[Theorem~4.7]{OPS15}}] \label{thm:ops}
  A $(k,n)$-noncrossing collection $\mathbb I$ of $k$-element subsets of $[n]$ is a maximal $(k,n)$-noncrossing collection if and only if $|\mathbb I| = k(n-k)+1$.
\end{thm}

\begin{rmk}
  This implies that Question \ref{question:main} is equivalent to:
  for which $(k,n)$ does there exist a symmetric noncrossing collection of cardinality $k(n-k)+1$?
  In Sections \ref{sec:specialconstruction} and \ref{sec:construction} we will construct collections of noncrossing sets and prove that they are maximal 
  by determining that they have the correct cardinality.
\end{rmk}

\section{Necessity of Condition~$\ref{cond:star}$}\label{sec:necessity}

The aim of this section is to prove the following statement:
\begin{prop}\label{prop:necessity}
If there exists a symmetric maximal $(k,n)$-noncrossing collection, then $(k,n)$ must satisfy Condition~$\ref{cond:star}$.
\end{prop}

The proof uses combinatorial tools developed in \cite{OPS15}, in particular a planar CW-complex which is associated to a noncrossing collection.
We recall some details about its construction for convenience.

Let $v_1, \dots, v_n$ be the vertices of a regular $n$-gon in $\R^2$ centered at the origin, labeled in clockwise order. 
Let $\left\{ e_1, \dots, e_n \right\}$ be the standard basis of $\R^n$, and define a linear map $p: \R^n\to \R^2$ by $p(e_i) = v_i$ for all $i$. 
If $I$ is any subset of $[n]$, we define $e_{I} = \sum _{i\in I}e_i\in \R^n$. 

Let now $\mathbb I$ be a maximal noncrossing $(k,n)$-collection. Denote by $V$ the set $\left\{ e_I\,|\, I\in \mathbb I \right\}\subseteq \R^n$, then we have 
$p(e_I) = \sum _{i\in I}v_i$.
This defines an embedding of $\mathbb I$ as a discrete collection of points in $\R^2$. There is a way of defining edges and faces so that we get a 
CW-complex $\Sigma(\mathbb I)$, which is also embedded in $\R^2$ and has $p(\mathbb I)$ as its set of vertices \cite[Proposition~9.4]{OPS15}.
Faces of $\Sigma(\mathbb I)$ are parametrised by some subsets of $[n]$ of cardinality $k-1$ and $k+1$.
Every edge is in the boundary of two faces, one corresponding to $k-1$ and one to $k+1$.
Moreover, we have that $\Sigma(\mathbb I)$ is homeomorphic to a disk \cite[Theorem~9.12 and Theorem~11.1]{OPS15}.

We will need the following lemma.

\begin{lemma}
  \label{lem:divides}
  Let $I$ be a subset of $[n]$ of size $l$, and assume that $I = I\+{n}k$. Then $d \mid l$, where $d = n/\on{GCD}(k,n)$.
\end{lemma}

\begin{proof}
  Observe that $d$ is the order of the function $\varphi:a\mapsto a+k$ on $\Z/n\Z$. Moreover, every element of $\Z/n\Z$ has the same order under $\varphi$. 
  Seeing $I$ as a subset of $\Z/n\Z$, we get then that $\varphi(I) = I$ and so $I$ is a union of $\varphi$-orbits. Since all $\varphi$-orbits have size $d$, 
  we get the claim.
\end{proof}

\begin{proof}[Proof of Proposition \ref{prop:necessity}.]
  If $k \in \left\{ 0, 1, n-1, n \right\}$, then Condition~$\ref{cond:star}$ is satisfied. Therefore, assume in the following that $1<k<n-1$.

  Since $\mathbb I$ is symmetric, we know that $\mathbb I = \mathbb I\+{n}k$. Denoting by $\rho$ the clockwise rotation by $\frac{2\pi k}{n}$ centered at the origin,
  we have by definition 
  $$p(e_{i\+{n}k}) = v_{i\+{n}k} = \rho(v_i) = \rho\left( p(e_i) \right).$$
  This in turn implies that $\rho(\Sigma(\mathbb I)) = \Sigma(\mathbb I)$. 
  Observe that $d = n/\on{GCD}(k,n)$ is the order of $\rho$, and that $1<d<n$.
  In particular, observe that $\Sigma(\mathbb I)$ must contain the origin since it is a disk.

  Let us look at the origin $\bar0 \in \R^2$, which is the only fixed point of $\rho$. Three cases can happen:
  \begin{enumerate}
    \item $\bar0 = p(e_I)$ for some $I\in \mathbb I$. Then $I = I\+{n}k$, which implies that $d$ divides $|I| = k$ by Lemma \ref{lem:divides}. So
      Condition~$\ref{cond:star}$ is satisfied in this case.
    \item $\bar 0$ is not a vertex of $\Sigma(\mathbb I)$, but it lies on an edge. This can only happen if $d = 2$, but in this case the two faces having $\bar 0$ on their boundary cannot 
      be sent to each other by $\rho$ since their parameters have different cardinalities. So this case does not occur.
    \item $\bar 0$ is in the interior of a face $F$ of $\Sigma(\mathbb I)$. In this case we must have that $\rho(F) = F$, which implies that the vertices of $F$ are permuted by $\rho$. 
    These lie on a circle centered at $p(K)$, where $K$ is the label of $F$ \cite[Theorem~5.7(a)]{BKM16}.
      It follows that $\bar 0$ must be the center of this circle, so $\bar 0 = p(K)$ and thus $K = K\+{n}k$ since $\rho(\bar 0) = \bar 0$. This implies 
      that $d$ divides $|K|$ by Lemma \ref{lem:divides}. Since $|K|$ is either $k-1$ or $k+1$ being the label of $F$, Condition~$\ref{cond:star}$ is satisfied.
  \end{enumerate}
  So in all possible cases Condition~$\ref{cond:star}$ holds, hence the claim is proved.
\end{proof}

\section{Setup and notation}\label{sec:setup}
In this section we fix notation and prove some auxiliary results which we will need later. 
If $X$ is cyclically ordered and $a, b\in X$, define the set $[a,b]\subseteq X$ to be the smallest subset such that $a, b\in [a,b]$ and 
$S_X\left( [a, b]\setminus\left\{ b \right\} \right)\subseteq [a,b]$. We call a set of this type an \emph{interval} of $X$. 
Observe that $[a,b] = X$ if and only if $S_X(b) = a$, and otherwise $a$ and $b$ are uniquely determined by $[a,b]$. If $I\neq X$ we call $I$ a \emph{proper interval}.
We will write $[a,b]_X$ instead of $[a,b]$ to specify the set $X$ if needed.

For every $i\in X$, there is an associated linear order $<_i$ on $X$ defined by
\begin{align*}
  i <_i S_X(i)<_i S_X^2(i) <_i \cdots <_i S_X^{-1}(i).
\end{align*}
Observe that $a<_ib<_ic$ implies $a<_\circ b<_\circ c$ for all $a,b,c\in X$. There is a bijection between linear orders on $X$ with 
this property on the one hand and elements $i$ of $X$ on the other hand.
If $I =[a,b]\subseteq X$ is a proper interval, the \emph{linear order associated to $I$} is the order $<_{a}$.

\begin{lemma}\label{lem:interval}
  Let $I$ be a proper interval of $X$, and let $<_I$ be the linear order associated to $I$. If $a,c\in I$ and $b \in X$ are such 
  that $a<_Ib<_Ic$, then $b\in I$.
\end{lemma}

\begin{proof}
  Immediate from the definition of interval and of $<_I$.
\end{proof}

\begin{ex}
  Let $n = 8$. Then $I = \left\{ 7,8,1,2,3 \right\}$ is an interval of $[n]$. The linear order associated to $I$ is
  $7<8<1<2<3<4<5<6$. Let $Q = \{1,3,4,6,7\}$. Then $I\cap Q = \left\{ 7,1,3 \right\}$ is an interval of $Q$.
\end{ex}

\begin{lemma}\label{lem:intnoncross}
  Let $X$ be cyclically ordered, and let $I$ be an interval of $X$. Let $a<_\circ b<_\circ c<_\circ d \in X$ with $a,c\in I$. 
  Then $b\in I$ or $d\in I$.
\end{lemma}

\begin{proof}
  If $I = X$ then $b\in I$ and we are done. Otherwise, let $<_I$ be the linear order associated to $I$. Assume $a<_I c$. If $b<_I a$, 
  then $b<_\circ a<_\circ c$, contradiction. If $c<_I b$, then $a<_\circ c<_\circ b$, contradiction. So we must have $a<_Ib<_Ic$ hence
  $b\in I$. In the same way we obtain that $d\in I$ if we assume that $c<_Ia$.
\end{proof}

\begin{lemma}
  Let $X$ be a cyclically ordered set, $I$ an interval of $X$, and $J\subseteq X$. Then $I$ and $J$ are noncrossing.
\end{lemma}

\begin{proof}
  Since $I$ is an interval, if $a<_\circ b<_\circ c<_\circ d\in X$ with $a, c\in I$, then either $d$ or $b$ are in $I$ by Lemma~
  \ref{lem:intnoncross}. So $I$ and $J$ cannot be crossing.
\end{proof}

\section{Construction: the case $n = dk$}\label{sec:specialconstruction}
In order to prove Theorem~\ref{thm:main}, we will construct a symmetric maximal $(k,n)$-noncrossing collection whenever $(k,n)$ satisfies Condition~$\ref{cond:star}$. 
The construction will be performed in two steps: first we will make the additional assumption that $k \mid n$, and in Section~\ref{sec:construction} we 
will show how one can get rid of this assumption.

We first give a construction of a symmetric maximal noncrossing collection when $n = dk$ and $(k,dk)$ satisfies Condition~$\ref{cond:star}$.
Observe that in this case we have $\on{GCD}(k,n) = k$, so $n/\on{GCD}(k,n) = d$.
By assumption, $k = dp+c$, with $c\in \left\{ -1,0,1 \right\}$. For $a \in [n]$, write
$\overline a =  (a +k\Z) \cap [n]$. 
Choose a total order $\overline {a_1} < \overline{a_2} <\cdots <\overline{a_{k}}$ on these congruence classes (for simplicity, 
we assume that $\{a_1, \ldots, a_k\}=\{1, \ldots, k\}$). We construct collections $\mathcal L_s$, for $1\leq s \leq k-p+1$, in the 
following way.

Call $P_s = [n]\setminus \bigcup_{i =1}^{s-1}\overline{a_i}$, considered as a cyclically ordered subset of $[n]$. 
For $1\leq h\leq d$, write $P_{s,h} = P_s \setminus \left\{ S_{\overline{ a_s}}^m (a_s)\,|\, h\leq m< d \right\}$.
For fixed $h$ and $i \in [S_{P_s}(a_s-k),a_s]$, define 
$$
I(i,h)  = \{ i, S_{P_{s,h}}(i), \ldots, S_{P_{s,h}}^{k-1}(i)\}.
$$
We are interested in $k$-element sets, and it is easy to see that $|I(i,h)| = k$ if and only if 
$|P_{s,h}|\geq k$ (in particular, the cardinality of $I(i,h)$ is independent of $i$). However, we note that different choices of $h$ 
may give rise to the same set $I(i,h)$. In particular it is clear that if $P_s$ is big (in comparison to $k$), then large 
values of $h$ will give the same set $I(i,h)$ for fixed $i$. Therefore, for given values of $i$ and $h$ we set $h^*$ to be the minimal 
$h'$ for which $I(i,h') = I(i,h)$. This element $h^*$ can be explicitly determined: we have that $h^* = |\overline{a_s}\cap I(i,h)|$, and that $h^*$ is the unique $h'$ such that
$S^{h'-1}_{P_{s,h}}(a_s)\in I(i,h)$.

Let $B_s$ be the collection defined by
$$
B_s = \left\{ I(i,h)\, |\, i \in [S_{P_s}(a_s-k),a_s],\, 1\leq h\leq d,\, |I(i,h)| = k\right\}.
$$
We define $\mathcal L_s = \left\{I\+{dk} xk\,|\,I \in B_s, x\in \Z \right\}$ and $\mathbb I = \bigcup_{s =1}^{k-p+1}\mathcal L_s$. 

\begin{rmk}
  There is a way of generating all the elements of $B_s$, which we explain informally. We start with $\{a_s\}$. We keep adding successors 
  in $P_s$ until we have a $k$-element set $I$ (which is an interval of $P_s$). This is our first set in $B_s$. If it contains an 
  element in $\overline{a_s}$ which is not $a_s$, then we can generate another set in $B_s$ by removing the last (in the order 
  $<_{a_s}$ of $\overline{a_s}$) such element, and adding another element at the end of $I$. However, we cannot add an element of 
  $\overline{a_s}$ in this way, so if the next element we would add is in $\overline{a_s}$ we skip it and add the next one instead.
  Thus we get another set in $B_s$. 
  
  Now if the latest set we constructed still has an element in $\overline{a_s}$ which is not $a_s$, we can remove the last such and add another element at the end, and thus produce another set in $B_s$.
  We can repeat this until we get a set $I$ such that $I\cap \overline{a_s} = \left\{ a_s \right\}$.

  Now we can start the whole construction again, beginning with $\left\{ S^{-1}_{P_s}(a_s), a_s \right\}$. We add successors until we have $k$ elements, and then we modify our resulting set by removing elements in $\overline{a_s}$ and
  adding elements at the end. We get some more sets in $B_s$ in this way. 
  Repeating this, starting with the various intervals $[S^{-x}_{P_s}(a_s), a_s]$, we get all the sets in $B_s$. Observe that the last $x$ we try is $x = k-1$.
\end{rmk}

\begin{ex}
  Let us illustrate the construction with a concrete example. Let $(k,n) = (7,28)$, so that $d = 4$ and Condition~$\ref{cond:star}$ is satisfied (this does not play a role here).
  We fix a total order $\bar 4< \bar 6<\bar 7<\bar 2<\bar 1<\bar 3<\bar 5$ on the congruence classes modulo $7$. We take $s = 4$, so that $a_s = 2$ and 
  $P_s = \left\{ 1,2,3,5,8,9,10,12,15,16,17,19,22,23,24,26 \right\}$.
  In Figure \ref{fig:example728} we draw the set $P_4$ on the circle, with crosses to indicate the elements of $[28]\setminus P_4$.
  The orbit $\overline{ a_s} = \bar 2 = \left\{ 2,9,16,23 \right\}$ is highlighted. 
  
  The arcs represent the $7$ elements of the set $B_4$, namely
  \begin{align*}
&\left\{ 24,26,1,2,3,5,8 \right\}, 
\left\{ 26,1,2,3,5,8,9 \right\},
\left\{ 26,1,2,3,5,8,10 \right\},
\left\{ 1,2,3,5,8,9,10 \right\},\\
&\hspace{1.3cm}  \left\{ 1,2,3,5,8,10,12 \right\},
  \left\{ 2,3,5,8,9,10,12 \right\},
  \left\{ 2,3,5,8,10,12,15 \right\}.
\end{align*}
Observe that they all contain $2$ and that some of them but not all contain $9$. 
The set $\mathcal L_4$ consists of all the shifts of the $7$ sets above by multiples of $7$ modulo $28$, and thus has $28$ elements.
The reader is invited to pick two sets in $B_4$ and check that they are noncrossing, and to do the same with two arbitrary sets in $\mathcal L_4$.
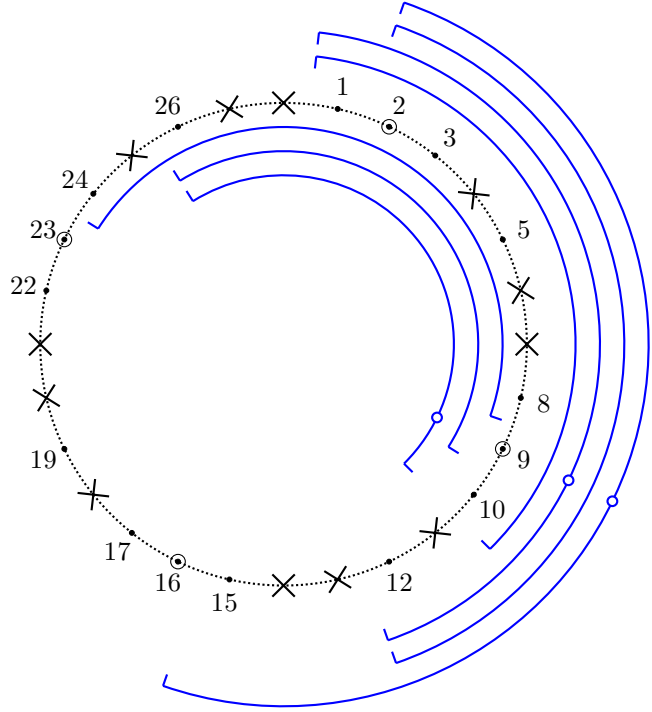
\begin{figure}
\[
\begin{tikzpicture}[baseline=(bb.base),scale=.8,
    cross/.style={cross out, draw, 
	 minimum size=2*(#1-\pgflinewidth), 
	 inner sep=0pt, outer sep=0pt},
	 arcsegment/.style = {blue, thick},
  ]

\newcommand{\nth}{360/28}
\newcommand{\radius}{4cm} 
\newcommand{\eps}{11pt} 
\newcommand{\dotrad}{0.04cm} 

\path (0,0) node (bb) {};


\draw (0,0) circle(\radius) [thick, densely dotted];


\foreach \n in {1,...,28}
{
\coordinate (b\n) at (-\nth*\n+90:\radius);
}


\foreach \n in {1,2,3,5,8,9,10,12,15,16,17,19,22,23,24,26}
{
  \draw (-\nth*\n+90:\radius + \eps) node {$\n$};
  \draw  (b\n) circle(\dotrad) [fill = black];
}


\foreach \n in {4,6,7,11,13,14,18,20,21,25,27,28}
{
  \draw (b\n) node[cross= 5pt, rotate =-\n*\nth, thick] {};
}

\foreach \n in {2,9,16,23}
{
\draw  (b\n) circle(\dotrad*3) [fill = none];
}


\draw [arcsegment](-1.5*\nth:\radius*.9) arc (-1.5*\nth:11.5*\nth:\radius*.9);
\draw [arcsegment] (-2.5*\nth:\radius*.8) arc (-2.5*\nth:9.5*\nth:\radius*.8);
\draw [arcsegment](-3.5*\nth:\radius*.7) arc (-3.5*\nth:9.5*\nth:\radius*.7);

\draw [arcsegment] (-3.5*\nth:\radius*1.2) arc (-3.5*\nth:6.5*\nth:\radius*1.2);
\draw [arcsegment] (-5.5*\nth:\radius*1.3) arc (-5.5*\nth:6.5*\nth:\radius*1.3);
\draw [arcsegment] (-5.5*\nth:\radius*1.4) arc (-5.5*\nth:5.5*\nth:\radius*1.4);
\draw [arcsegment] (-8.5*\nth:\radius*1.5) arc (-8.5*\nth:5.5*\nth:\radius*1.5);


\draw [arcsegment] (-1.5*\nth:\radius*.9)--(-1.5*\nth:\radius*.95);
\draw [arcsegment] (11.5*\nth:\radius*.9)--(11.5*\nth:\radius*.95);
\draw [arcsegment] (-2.5*\nth:\radius*.8)--(-2.5*\nth:\radius*.85);
\draw [arcsegment] (9.5*\nth:\radius*.8)--(9.5*\nth:\radius*.85);
\draw [arcsegment] (-3.5*\nth:\radius*.7)--(-3.5*\nth:\radius*.75);
\draw [arcsegment] (9.5*\nth:\radius*.7)--(9.5*\nth:\radius*.75);

\draw [arcsegment] (-3.5*\nth:\radius*1.2) -- (-3.5*\nth:\radius*1.15);
\draw [arcsegment] (6.5*\nth:\radius*1.2) -- (6.5*\nth:\radius*1.15);
\draw [arcsegment] (-5.5*\nth:\radius*1.3) -- (-5.5*\nth:\radius*1.25);
\draw [arcsegment] (6.5*\nth:\radius*1.3) -- (6.5*\nth:\radius*1.25);
\draw [arcsegment] (-5.5*\nth:\radius*1.4) -- (-5.5*\nth:\radius*1.35);
\draw [arcsegment] (5.5*\nth:\radius*1.4) -- (5.5*\nth:\radius*1.35);
\draw [arcsegment] (-8.5*\nth:\radius*1.5) -- (-8.5*\nth:\radius*1.45);
\draw [arcsegment] (5.5*\nth:\radius*1.5) -- (5.5*\nth:\radius*1.45);


\draw (-2*\nth:\radius*.7) circle(\dotrad*2) [arcsegment,fill = white];
\draw (-2*\nth:\radius*1.3) circle(\dotrad*2) [arcsegment,fill = white];
\draw (-2*\nth:\radius*1.5) circle(\dotrad*2) [arcsegment,fill = white];

 \end{tikzpicture}
\]
\caption{An example of our construction for $(k,n) = (7,28)$ and $s = 4$.}
\label{fig:example728}
\end{figure}

\end{ex}

Our claim is that $\mathbb I$ is a symmetric maximal $(k,n)$-noncrossing collection. 
To prove this, we will show that it consists of mutually noncrossing sets and that $|\mathbb I| = k(n-k)+1$. 
Thus we will be able to conclude that the claim holds using Theorem~\ref{thm:ops}. 
In this process, the only step that uses Condition~\ref{cond:star} is checking the cardinality of the last nonempty $\l_s$.

\begin{lemma}
  If $s\neq t$, then $\mathcal L_s \cap \mathcal L_t =  \varnothing$. 
\end{lemma}

\begin{proof}
  By symmetry, assume $s<t$. Since $\overline{a_s}\in I$ for every $I\in B_s$, every element in $\mathcal L_s$ contains some $a\in \overline{a_s}$.
  On the other hand, no element in $\mathcal L_t$
  contains such an element $a$.
\end{proof}

Now we count the number of elements in each collection $\l_s$.

\begin{prop}\label{prop:counting}
 \  
  \begin{enumerate}
    \item For all $s<k-p$, we have $|\mathcal L_s| = kd$.
    \item If $k\equiv -1$ or $k\equiv 0 \pmod d$, then $|\mathcal L_{k-p}| = kd$.
    \item If $k \equiv 1\pmod d$, then $|\mathcal L_{k-p}| = d(k-p)$.
  \end{enumerate}
\end{prop}

\begin{proof}

We start by remarking that in the cases we consider we have $|\l_s| = d|B_s|$. Indeed, let $I\in B_s$. Then $J = I\cap \overline{a_s}$ is an
interval of $\overline{a_s}$. The sets $I\+{n} kx$ intersect $\overline{a_s}$ in $J\+{n}kx$, and the only case 
in which these $d$ intervals of $\overline{a_s}$ are not all distinct is when $J = \overline{a_s}$, that is when $I$ is an interval of $P_s$. 
Now, by the same argument, the $d$ intervals $I+kx$ of $P_s$ are all distinct unless $I = P_s$.
In particular, if $|\l_s| \neq d|B_s|$, we must have $|P_s| = k$. 
Since $|P_s| = d(k-s+1)$, this can only happen (assuming Condition~\ref{cond:star}) if $k = dp$ with $p = k-s+1$. 
In particular, in the cases we consider in assertions $(1)$--$(3)$, we must have $|\l_s| = d|B_s|$.
As a consequence, it suffices to compute the cardinality of $B_s$.

To count the elements of $B_s$, we will identify a suitable domain that makes the function $(i,h)\mapsto I(i,h)$ injective, and
count the elements of this domain.

Recall that to each pair $(i,h)$ we can assign the pair $(i,h^*)$, where $h^*$ is such that $I(i,h) = I(i,h^*)$ and minimal with respect to this
property.

On the other hand, if $|P_{s,h}|>k$, we can recover $i$ from the set $I(i,h)$. This is because the set $I(i,h)$ is an interval of 
$P_{s,h}$, and $i$ is its starting element.

If $|P_{s,h}| = k$, then the sets $I(i,h)$ are all equal to $P_{s,h}$, so these cases will require special attention. 

Let $r = k - s + 1$. Observe that there exists a unique bijection that preserves cyclic order from $P_s$ to $[dr]$ 
such that $a_s$ is sent to $r$. Therefore we will assume in the following that $P_s = [dr]$. 

It will be convenient to fix $i\in [1, r]$ and let $h$ vary.
We want to count, for a fixed $i$, how many sets $I(i,h)$ there are such that $h = h^*$ (to avoid double counting). 
In other words, how many sets $I(i,h)$ there are such that $rh \in I(i,h)$.
The set $I(i,h)$ must contain the interval $[i,rh]$, so from $|I|= k$ we get $k \geq rh-i+1$.
Setting 
$$\gamma_i = \left\lfloor \frac{k+i-1}{r}\right\rfloor,$$ 
we obtain $h\leq \gamma_i$. We conclude that for a fixed $i\in [1,r]$ there are exactly $\gamma_i$ distinct values of $h$ 
such that $I(i,h)$ has size $k$, and thus $\gamma_i$ elements in $B_s$. 

As we pointed out, one can recover $i$ from $I(i,h)$ unless $|P_{s,h}| = k$, which means that 
$|B_s|= \sum_{i=1}^r\gamma_i$ unless $|P_{s,h}| = k$.
Let us analyse the special case $|P_{s,h}| = k$, distinguishing between the three congruences permitted by Condition~\ref{cond:star}.

By construction, $|P_{s,h}| = dk-ds + h$, and recall that $1\leq h\leq d$. Now there are three cases, assuming $|P_{s,h}|= k$:
\begin{itemize}
  \item If $k = dp+1$, then $h= 1$ and $p =k-s$.
  \item If $k = dp$, then $h = d$ and $p = k-s+1$. 
  \item If $k = dp-1$, then $h = d-1$ and $p = k-s+1$.
\end{itemize}

To prove assertions $(1)$ and $(2)$ it is thus enough to show that 
\begin{align*}
  \sum_{i=1}^r\gamma_i =  k.
\end{align*} 
To prove this we write $k = ar + b$ where $a \in \Z$ and $0 \leq b < r$. Then, for $0 \leq j \leq r-1$, we have
 \begin{displaymath}
  \left\lfloor\frac{b+j}{r}\right\rfloor = \begin{cases}
                                            0, & \text{ if } j <r-b;\\
                                            1, & \text{ else.}
                                           \end{cases}
 \end{displaymath}
 Thus we obtain
 \begin{align*}
  \sum_{i=1}^r \gamma_i & = \sum_{i=1}^r \left\lfloor\frac{k+i-1}{r}\right\rfloor 
  \sum_{j=0}^{r-1} \left\lfloor\frac{k+j}{r}\right\rfloor  \\
  & = ar + \sum_{j=0}^{r-1}\left\lfloor\frac{b+j}{r}\right\rfloor = ar + b  \\
  & = k.
 \end{align*}

It remains to prove assertion $(3)$, so let us assume $k = dp+1$. Now it is convenient to fix $h$ and let $i$ vary.
For $h\geq 2$, we have $|P_{s,h}|>k$, which implies that we can recover $i$ from the set $I(i,h)$, which means that we 
can count as above and obtain $\sum_{i = 1}^{r}(\gamma_i-1)$ sets.
On the other hand, we also get one additional set $I(i,1) = P_{s,1}$ when $h = 1$. 
The cardinality of $B_s$ is thus
\begin{align*}
  |B_s| = 1+ \sum_{i = 1}^{r}(\gamma_i-1) = 1+ k-r = s = k-p
\end{align*}
as we claimed.

\end{proof}

Observe that if $h = d$, the sets $I(i,h)$ are actually intervals of $P_s$. In particular, $\l_s$ contains all the intervals of length $k$ of $P_s$.

\begin{prop}\label{prop:finalcounting}
  \
  \begin{itemize}
    \item If $k\equiv -1 \pmod d$, then $|\mathcal L_{k-p+1}| = k+1$.
    \item If $k\equiv 0 \pmod d$, then $|\mathcal L_{k-p+1}| = 1$.
    \item If $k\equiv 1 \pmod d$, then $\mathcal L_{k-p+1} = \varnothing$.
  \end{itemize}
\end{prop}

\begin{proof}
For $s = k-p+1$, the set $P_s$ has cardinality $dp$. If $k = dp -1$, then $P_s$ has $k+1$ distinct $k$-element subsets, and they are all in $\mathcal L_s$ since they are intervals of $P_s$.
  If $k = dp$, then $P_s$ itself is its only $k$-element subset, and it is in $\mathcal L_s$ since it is an interval of $P_s$.
  Finally, if $k = dp+1$, then $P_s$ has no $k$-element subsets.
\end{proof}

\begin{prop}\label{prop:cardinality}
  We have that $|\mathbb I| = k(dk-k)+1$.
\end{prop}

\begin{proof}
  There are three cases to consider, depending on the congruence class of $k$ modulo $d$. In each case, we will combine the results of Proposition \ref{prop:counting} and Proposition \ref{prop:finalcounting}.

     If $k\equiv -1\pmod d$, then 
      \begin{align*}
	|\mathbb I| &= kd(k-p) + k+1 = (dp-1)d(dp-1-p)  \\
	 &= d^3p^2-2d^2p-d^2p^2+d+dp + dp  \\
	 &= (dp-1)^2(d-1)+1 = 	k(dk-k)+1.
      \end{align*}
    
      If $k\equiv 0\pmod d$, then 
      \begin{align*}
	|\mathbb I| = kd(k-p)+1  = d^2p(dp-p)+1 = k(dk-k)+1.
      \end{align*}
    
      If $k\equiv 1\pmod d$, then 
      \begin{align*}
	|\mathbb I| &= kd(k-p-1) +d(k-p) = kd(k-p)-dp  \\
	&= (dp+1)d(dp-p+1)-dp = d^3p^2+2d^2p-d^2p^2-2dp +d  \\
	&= (dp+1)^2(d-1)+1 = k(dk-k)+1.
      \end{align*}
\end{proof}

Proposition~\ref{prop:cardinality} shows that $\mathbb I$ is a collection of $k$-element subsets of $[n]$ which has the cardinality of a maximal noncrossing collection, 
and recall that by construction $\mathbb I$ is symmetric. Thus it remains to prove that the sets in $\mathbb I$ are pairwise noncrossing.

We start with an immediate consequence of the discussion we already used to count the elements of $\l_s$:

\begin{lemma}
  \label{lem:wlog}
  Assume that $ I = I(i_1, h_1)\+{n}kx_1 = I(i_2, h_2)\+{n}kx_2\in \l_s$, with $h_1$ and $h_2$ chosen to be minimal. 
  Assume moreover that $I\neq P_s$. 
  Then $(i_1, h_1)= (i_2,h_2)$ and $x_1 \equiv x_2\pmod n$.
\end{lemma}

\begin{proof}
  First, we have that $h_1 = h_2 = |I\cap \overline{a_s}|$. The set $A= I\cap \overline{a_s}$ is an interval of $\overline{a_s}$. 

  If $A= \overline{a_s}$, then $I$ is a  proper interval of $P_s$, since it is not equal to $P_s$ by assumption. Then $i_1 = i_2$ is the minimal 
  element of this interval. There is moreover a unique (modulo $n$) integer $x$ such that $i_1\in [S^{x-1}_{\overline{a_s}}(a_s), S_{\overline{a_s}}^{x}(a_s)]$, 
  and this has to coincide with both $x_1$ and $x_2$. 
  
  If $A\neq \overline{a_s}$, then
  $A$ has a minimal element which is equal to both $S_{\overline{a_s}}^{x_1}(a_s)$ and $S_{\overline{a_s}}^{x_2}(a_s)$, so we obtain that $x_1 \equiv x_2\pmod n$.
  Now $i_1$ and $i_2$ are both equal to the minimal element of $[S^{x_1-1}_{\overline{a_s}}(a_s), S_{\overline{a_s}}^{x_1}(a_s)]\cap I$, so they coincide.
\end{proof}

In view of Lemma~\ref{lem:wlog}, we will often be able to reduce to considering only the case $I= I(i,h)\in B_s$.
In the rest of this section, we will always assume that the parameter $h$ is chosen to be minimal.

\begin{lemma}\label{lem:nojump}
  Let $I\in \mathcal L_s$.
  Let $a,c\in I$, and let $b,d\in P_s\setminus I$ such that $a<_\circ b<_\circ c<_\circ d$. Then $b\in \overline{a_s}$ or 
  $d\in \overline{a_s}$.
\end{lemma}

\begin{proof}
  By Lemma~\ref{lem:wlog}, we can assume that $I = I(i,h)\in B_s$. 
  Thus $I$ is an interval of $P_{s,h}$, so by Lemma~\ref{lem:intnoncross} we deduce that $b$ or $d$ has to lie in $P_{s}\setminus P_{s,h}\subseteq \overline{a_s}$.
\end{proof}

The following two propositions show that the elements of $\mathbb I$ are noncrossing.
\begin{prop}\label{prop:noncross1}
  If $I\in \mathcal L_s$ and $J\in \mathcal L_t$ for $s\neq t$, 
  then $I$ and $J$ are noncrossing.
\end{prop} 

\begin{proof}
  By symmetry, assume $s<t$. Then $J\subseteq [n]\setminus \bigcup_{i =1}^{s}\overline{a_i}$.
  Assume to reach a contradiction that $I$ and $J$ are crossing. Then there are $a,c\in I\setminus J$ and 
  $b,d\in J\setminus I$ with $a<_\circ b<_\circ c<_\circ d$. By Lemma \ref{lem:nojump}, it follows that $b$ or $d$ are in $\overline{a_s}$, which is a contradiction 
  since $b,d \in J$.
\end{proof}

\begin{prop}\label{prop:noncross2}
If $I,J\in \mathcal L_s$, then $I$ and $J$ are noncrossing.
\end{prop}

\begin{proof}
  Assume that $I,J$ are crossing, that is, there exist $a<_\circ b<_\circ c<_\circ d\in P_s$ with $a,c\in I\setminus J$ and $b,d\in J\setminus I$. By
  applying Lemma \ref{lem:nojump} to first $I$ and then $J$, we can without loss of generality assume $a, b\in \overline{a_s}$.
  By Lemma~\ref{lem:wlog}, we can assume that $J = I(j,h)\in B_s$. Observe that $I\neq J$ by assumption, so $|P_s| >k$ and $j$ is uniquely determined.
  Let us consider the linear order $<_j$ with minimal element $j$ on $P_s$. If $d<_j b$, then by construction the interval $[d,b]_{P_s}$ is contained in $J$. 
  However, we have that $a\in [d,b]_{P_s}\setminus J$, therefore this cannot happen and we 
  must then have $b<_j d$. In particular $c\in [b,d]_{P_s}\subset J\cup \overline {a_s}$, so we conclude that $c\in \overline{a_s}$.

  We thus have that $a, b, c\in \overline{a_s}$, with $a, c\in I$. By construction, since $I$ in $\l_s$, we must then have $b\in I$, a contradiction.
\end{proof}

\begin{thm}
  \label{thm:dk}
  If $(k,dk)$ satisfies Condition~$\ref{cond:star}$, then the collection $\mathbb I$ constructed above is a symmetric maximal $(k,dk)$-noncrossing
  collection.
\end{thm}
\begin{proof}
  By construction, $\mathbb I$ is a collection of $k$-element subsets of $[dk]$. It is symmetric since $\l_s = \l_s\+{dk}k$ for every 
  $s$. If $I,J \in \mathbb I$, then $I$ and $J$ are noncrossing by Proposition \ref{prop:noncross1} and Proposition \ref{prop:noncross2}. 
  Finally, Proposition \ref{prop:cardinality} and Theorem \ref{thm:ops} imply that $\mathbb I$ is a maximal noncrossing collection.
\end{proof}

\section{Construction: the general case}\label{sec:construction}
Now assume that $(k,n)$ are integers which satisfy Condition~$\ref{cond:star}$. 
Set $g = \on{GCD}(k,n)$ and $d = n/g$. Since $\on{GCD}(k,dk) = k$, Condition~$\ref{cond:star}$ is satisfied for $(k,dk)$ .
We will first construct a symmetric maximal $(k,dk)$-noncrossing collection, then pick a suitable subcollection which will be in bijection 
with a maximal $(k,n)$-noncrossing collection. 

Choose any linear order on the classes $\bar  1, \dots ,\bar g$ modulo $n$.
Complete it to a linear order on all classes $\bar 1, \dots, \bar k$ modulo $n$ such that 
\begin{align*}
  \on{min}\left\{ \bar  1, \dots ,\bar g \right\}> \on{max}\left\{ \overline{g+1}, \dots, \bar k \right\}.
\end{align*}
Construct a symmetric maximal $(k,dk)$-noncrossing collection $\mathbb I$ as in Section \ref{sec:specialconstruction} with this linear 
order as datum. Define 
\begin{align*}
  \mathbb J  = \mathbb I\setminus \bigcup_{s = 1}^{k-g}\mathcal L_s.
\end{align*}

Thus $\mathbb J$ is a collection of $k$-element subsets of $A = \bigcup_{s = 1}^{g}\overline{a_s}$. Notice that $A$ has a cyclic order, 
induced by that on $[dk]$. For an example of this construction see Section \ref{sec:examples}.

Observe that, for any $s\in [g]$, the cardinality of $\overline{a_s}$ is the order of $k$ in $\Z/n\Z$, which is $n/g =d$.
It follows that $|A|= gd = n$.

\begin{prop}
  \label{prop:J}
  The collection $\mathbb J$ is a maximal noncrossing collection of $k$-element subsets of $A$ which satisfies $\mathbb J=\mathbb J\+{dk}k $.
\end{prop}

\begin{proof}
  The collection $\mathbb J$ is by construction a collection of noncrossing subsets of $A$, since the cyclic order on $A$ is induced by that on $[dk]$.
  For every $s$ we have that $\mathcal L_s = \mathcal L_s\+{dk}k$ by construction, hence $\mathbb J = \mathbb J\+{dk}k$ since $\mathbb J$ is a union of various $\mathcal L_s$.
  We can compute
  \begin{align*}
    |\mathbb J| = |\mathbb I|-\sum_{s = 1}^{k-g}\mathcal L_s = 
    	 k(dk-k)+1 -(k-g)dk = dkg-k^2+1 = 
	 k(n-k)+1.
  \end{align*}
  Since $|A| = n$, we conclude using Theorem \ref{thm:ops} that $\mathbb J$ is maximal among all collections of noncrossing $k$-element subsets of $A$.
\end{proof}

Now define a function $F:[n]\to A$ by $$F(a+g x) = a+kx$$ for $a\in [g]$ and $x =0, \dots, d-1$. This is well defined and injective by the division algorithm on $[n] $ and $[dk]$ respectively.
Since $|A| = n$, we conclude that $F$ is bijective.

\begin{lemma}\label{lem:cyclicorder}
  We have $F\circ S_{[n]} = S_A\circ F$.
\end{lemma}

\begin{proof}
  The function $F$ is increasing, so it preserves the linear orders on $[n]$ and $A$ with minimum elements equal to 1.
  Since $F(1) = 1$, we conclude that $S_A(F(x)) = F(x+1)$ for all $1\leq x\leq n-1$.
  But $S_A(g+ k(d-1)) = 1 = F (n +1)$, which proves the claim.
\end{proof}

We extend $F$ to subsets of $[n]$ and still call it $F$. It is a bijection between subsets of $[n]$ and subsets of $A$.
Call $\mathbb I' = F^{-1}(\mathbb J)$.

\begin{prop}\label{prop:I'symm}
  We have $\mathbb I' = \mathbb I'\+{n}k$.
\end{prop}

\begin{proof}
  Pick $I\in \mathbb I'$. Then by Proposition \ref{prop:J} we have $F(I)\+{dk}k\in \mathbb J$. Then 
  \begin{align*}
    F^{-1}(F(I)\+{dk}k) = I\+{n}g\in \mathbb I'
  \end{align*}
  so that $\mathbb I' = \mathbb I'\+{n}g$. Since $k = g\cdot \frac{k}{g}$ and $\frac{k}{g}$ is an integer, we are done.
\end{proof}

\begin{thm}
  \label{thm:existence}
  If $(k,n)$ satisfies Condition~$\ref{cond:star}$, then the collection $\mathbb I'$ constructed above is a symmetric maximal $(k,n)$-noncrossing collection.
\end{thm}

\begin{proof}
  By construction, $\mathbb I'$ is a collection of $k$-element subsets of $[n]$. It is symmetric by Proposition~\ref{prop:I'symm}.
  It is a maximal $(k,n)$-noncrossing collection by Proposition \ref{prop:J} and Lemma \ref{lem:cyclicorder}.
\end{proof}

We can now prove Theorem \ref{thm:main}. 

\begin{proof}[Proof of Theorem \ref{thm:main}]
  By Proposition \ref{prop:necessity}, Condition~$\ref{cond:star}$ is necessary.
  By Theorem \ref{thm:existence}, the collection constructed in Section \ref{sec:construction} is a 
  symmetric maximal $(k,n)$-noncrossing collection, so Condition~$\ref{cond:star}$ is also sufficient.
\end{proof}

\begin{rmk}
  There exist symmetric maximal noncrossing collections which do not come from our construction. Indeed, one can produce $\on{GCD}(k,n)!/(p-1)!$ different collections by varying the 
  total order on $\left\{ \overline{a_1}, \dots, \overline{a_k} \right\}$ to define the various $\l_s$. However, a computer search produces the lower bounds for the number of distinct 
  symmetric maximal $(k,n)$-noncrossing collections shown in the following table.
\begin{center}
  \begin{tabular}{ | c | c | c | c | c | c | c | c | c | c |  c | c | c | c | c | c| }
    \hline
    $k\setminus n$ & 4 & 5& 6&7&8&9&10&11&12&13&14&15&16&17&18 \\ \hline    2 & 2 & & 2  &&&&&&&&&&&&\\ \hline
    3 &&& 6 &&& 24 &&& 24 &&&&&&\\\hline
    4 &&&  && 110 && 6 && $>894$ &&&& $>1900$ &&\\\hline
 5&&&  &&&  & $>2000$ && &&& $>4800$ &&&\\\hline
 6&&&  &&&  &&& $>4900$ && 18 & $>840$ &&& $>5000$\\\hline
 7&&&  &&&  &&& && $>5000$ &&&&\\\hline
 8&&&  &&&  &&& &&& & $>5000$&& 54\\\hline
  \end{tabular}
\end{center}
\vspace{.2cm}
The numbers without a $>$ sign are exact. It is easy to check by hand that, for instance, our construction produces only 2 of the 6 symmetric maximal $(4,10)$-noncrossing collections.
\end{rmk}

\section{Example}\label{sec:examples}

In this section we illustrate our construction with an example. Suppose we want to construct a symmetric maximal $(4,10)$-noncrossing collection.
Since $n=10$ is not a multiple of $k=4$, we need to use the procedure described in Section~\ref{sec:construction}. Thus we will first construct a 
symmetric maximal $(4,20)$-noncrossing collection, where $g = 2$, $d=5$, and $20 = kd$.

We pick the following order on congruence classes modulo $20$: 
$$
\bar 3<\bar 4<\bar 2<\bar 1.$$ 
Then we get the set $$B_1 = \{ \{20,1,2,3\}, \{1,2,3,4\}, \{2,3,4,5\}, \{3,4,5,6\}\}.$$ 
Notice that all these sets contain the element $3$. 
The set $\l_1 = B_1\+{20} 4\Z $ consists of the $20$ intervals of $[20]$.

The next step is removing the orbit $\bar 3 = \{3,7,11,15,19\}$ and constructing 
$$B_2= \{ \{1,2,4,5\},\{2,4,5,6\},\{4,5,6,8\}, \{4,5,6,9\}\}.$$
Again, we define $\l_2 = B_2\+{20}4\Z  $.

Next we remove the orbit of $4$ and construct
$$
B_3 = \{\{1,2,5,6\}, \{1,2,5,9\}, \{2,5,6,9\}, \{2,5,9,13\}\}$$
and $\l_3 = B_3\+{20}4\Z$.

Finally, when the only orbit left is $\bar 1 = \{1,5,9,13,17\}$ we get 
$$
B_4 = \{\{1,5,9,13\}\}$$
so that $\l_4$ consists of the $5$ different shifts of $\{1,5,9,13\}$.

By Theorem~\ref{thm:dk}, the collection $\mathbb I = \l_1\cup \l_2\cup \l_3\cup\l_4$ is a symmetric maximal $(4,20)$-noncrossing collection.
Now we restrict our attention to $\mathbb J = \l_3\cup \l_4$, which is a maximal noncrossing collection in the cyclically ordered set $\bar 1\cup \bar 2\subseteq [20]$.
Observe that $\mathbb J$ has $25$ elements, which is indeed $4(10-4)+1$ (cf.~Theorem~\ref{thm:ops}).
It remains to rename the elements of $ \mathbb J$ to obtain a symmetric maximal noncrossing collection in $[10]$. 
The function $F$ defined in Section~\ref{sec:construction} maps $(1,2,3,4,5,6,7,8,9,10)$ to $(1,2,5,6,9,10,13,14,17,18)$, so 
the collection $\mathbb I' = F^{-1}(\l_3\cup\l_4)$ is 
$$
\{\{1,2,3,4\},\{1,2,3,5\},\{2,3,4,5\},\{2,3,5,7\},\{1,3,5,7\}\}\+{10}2\Z.
$$
As per Theorem~\ref{thm:existence}, this is a maximal $(4,10)$-noncrossing collection invariant under adding $2$ modulo $10$, and thus invariant 
under adding $4$ modulo $10$ as we wanted.

\section{Algebraic consequences}\label{sec:algcons}

In this section we illustrate some representation-theoretic consequences of Theorem \ref{thm:main}.

There is a way of generating a quiver $Q$ (i.e.,~a directed graph) embedded in $\R^2$, starting from a maximal $(k,n)$-noncrossing collection $\mathbb I$.
As a graph, $Q$ is nothing but the 1-skeleton of the CW-complex $\Sigma(\mathbb I)$ of Section~\ref{sec:necessity}. The edges are oriented such that the 
face to the right is the one whose label has size $k-1$. 
As was observed in Section~\ref{sec:necessity}, this quiver is invariant under rotation by $\frac{2\pi k}{n}$ if and only if $\mathbb I$ is symmetric.

One can define a potential $W$ on $Q$ by taking the sum of all clockwise faces minus the sum of all anticlockwise faces. Thus $(Q, W)$ becomes a quiver with potential 
in the sense of \cite{DWZ08}. By taking the boundary vertices (those corresponding to the intervals of $[n]$) as frozen, one can then define the frozen Jacobian algebra $A$ \cite[Definition~1.1]{BIRS11}.

The main result of \cite{BKM16} is that $A\cong \on{End}_B(T)$, where $B = B(k,n)$ is an infinite-dimensional algebra introduced in \cite{JKS16} and $T$ is a cluster tilting object in the category $\on{CM}(B)$. 
In fact, one can associate a rank one Cohen-Macaulay $B$-module $L_I$ to each $I\subseteq [n]$ of size $k$ (see \cite[\S5]{JKS16}), and take $T= \bigoplus_{I\in \mathbb I}L_I$. 
It is indeed proved in \cite[Proposition~5.6]{JKS16} that $I$ and $J$ are noncrossing if and only if $\on{Ext}^1(L_I, L_J)= 0 = \on{Ext}^1(L_J, L_I)$.

One can also look at the stable category $\underline{\on{CM}}(B)$ of $\on{CM}(B)$, which is triangulated and in fact 2-Calabi-Yau \cite[Corollary~4.6]{JKS16} and \cite[Proposition~3.4]{GLS08}.
The triangulated structure of this category plays a crucial role in the motivations behind this article: we have that $L_I[-2]\cong L_{I\+{n}k}$ in $\underline{\on{CM}}(B)$ \cite[Proposition~2.7]{BB17},\cite[\S7]{JKS16}.
Thus saying that $\mathbb I= \mathbb I\+{n}k$ is equivalent to saying that $T\cong T[2]$.

\begin{rmk}
  In \cite{BKM16} and \cite{Pas17b}, the focus is on Postnikov diagrams (or alternating strand diagrams). By \cite[Theorem~11.1]{OPS15}, there is a bijection between Postnikov diagrams and 
  maximal noncrossing collections, so the two concepts are interchangeable. In this article we focus on collections since constructing a maximal noncrossing collection explicitly is much easier than 
  constructing a Postnikov diagram.
\end{rmk}

If we denote by $e\in A$ the idempotent corresponding to the boundary vertices of $Q$, we have that $A/AeA\cong \on{End}_{\underline{\on{CM}}(B)}(T)$ by \cite[Lemma~6.5]{Pas17b}.
One can prove (see \cite[Proposition~4.2]{Pas17b}) that $A/AeA$ is self-injective if and only if $T\cong T[2]$, which as we saw holds if and only if $\mathbb I$ is symmetric.
The algebra $\Lambda = \Lambda(\mathbb I)= A/AeA$ is the Jacobian algebra of the quiver with potential obtained from $(Q,W)$ by removing the boundary vertices.
Observe that these correspond to the intervals of $[n]$, which are part of all maximal noncrossing collections. Hence all the information carried by $\mathbb I$ is 
preserved by looking only at the sets in $\mathbb I$ which are not intervals.

The original motivation of this work was to find examples of self-injective Jacobian algebras. This interest stems in turn from higher dimensional Auslander-Reiten theory, in which these algebras play a role analogous 
to that of preprojective algebras of Dynkin quivers (see \cite{HI11b}). One consequence of Theorem \ref{thm:main} is that there exist many such algebras (which is a priori unclear, cf.~\cite[Question~10.1]{HI11b}).

\begin{cor}\label{cor:algcons1}
  Let $(k,n)$ be a pair satisfying Condition~$\ref{cond:star}$, and let $B = B(k,n)$ be the algebra defined in \cite[\S3]{JKS16}. Then:
  \begin{enumerate}
    \item there exists a cluster tilting module $T\cong T[2]\in \underline{\on{CM}}(B)$ whose indecomposable summands are rank one modules;
    \item the algebra $\Lambda = \on{End}_{\underline{\on{CM}}(B)}(T)$ is a self-injective Jacobian algebra;
    \item a Nakayama automorphism of $\Lambda$ is induced by $I\mapsto I-k$.
  \end{enumerate}
\end{cor}

\begin{proof}
  By Theorem \ref{thm:main}, there exists a maximal $(k,n)$-noncrossing collection $\mathbb I = \mathbb I\+{n}k$. Take $T = \bigoplus_{I\in \mathbb I}L_I$, then the statements follow from 
  \cite[Theorem~7.2]{Pas17b}.
\end{proof}

Another interesting consequence of Theorem \ref{thm:main} comes from looking at the order of the Nakayama automorphism given in Corollary $\ref{cor:algcons1}(3)$.
Suppose we fix $d$ and we want to construct examples of self-injective Jacobian algebras with Nakayama automorphism of order $d$. One possibility that works for every $d\geq3$ is to take 
a cluster tilted algebra \cite{BIRS11}, \cite{Kel11}. A result by Ringel \cite{Rin08} classified the self-injective ones, and it turns out that for a fixed $d$
there are at most two such algebras with Nakayama automorphism of order $d$.
All the other examples presented in \cite{HI11b} have Nakayama automorphism of order $2$ or $3$ (for these, infinite families are given). Sporadic examples of order $4$ and $5$ have later been found by 
Herschend and Lakani independently. An example of a self-injective Jacobian algebra with Nakayama automorphism of order $2x+1$ for any $x\in \Z_{>0}$ is constructed in \cite{Pas17b} using Postnikov diagrams.
As a corollary of our construction, we get:
\begin{cor}\label{cor:algcons2}
  Let $d\in \Z_{>1}$. There exist infinitely many self-injective Jacobian algebras with a Nakayama automorphism of order $d$.
\end{cor}

\begin{proof}
  Choose $d\in \Z_{>1}$. There are infinitely many choices of $(k,n)$ satisfying Condition~$\ref{cond:star}$ such that $d = n/\on{GCD}(k,n)$. Indeed, take for instance
 $$
 (k,n) \in \left\{ (d,d^2), (2d, 2d^2), (3d, 3d^2), \dots\right\}.
 $$
 By Corollary \ref{cor:algcons1}, for each of these pairs there exists a self-injective Jacobian algebra with Nakayama automorphism of order equal to the order of $a\mapsto a-k$ on $\Z/n\Z$, that is, $d$.
 They are pairwise non-isomorphic, since their quivers have different numbers of vertices.
\end{proof}

\begin{rmk}
  One can also choose the families 
  $$
  (k,n) \in\left\{ (d\pm1,d(d\pm 1)), (2(d\pm 1), 2d(d\pm 1)), (3(d\pm 1), 3d(d\pm 1)), \dots\right\}
$$
in the proof of Corollary \ref{cor:algcons2}, to get other infinite families of self-injective Jacobian algebras with Nakayama automorphism of order $d$. In the latter families, the 
Nakayama permutation acts freely on the vertices of the quiver, while in the family used in the proof there is a fixed vertex.
\end{rmk}

\bibliographystyle{alpha}
\bibliography{Bibliography}

\end{document}